\documentclass[11pt,a4paper]{article}

\usepackage{mathptmx}
\usepackage{amsmath,amssymb,amsfonts,amsthm}
\usepackage[english]{babel}
\usepackage{tikz,float}
\usepackage{enumerate}
\usepackage{url}
\usepackage{a4}
\theoremstyle{plain}
\newtheorem{theorem}{Theorem}[section]
\newtheorem{lemma}[theorem]{Lemma}
\theoremstyle{remark}
\newtheorem{remark}[theorem]{Remark}
\theoremstyle{definition}
\newtheorem{assumption}{Assumption}

\newtheorem{definition}[theorem]{Definition}
\newtheorem{example}[theorem]{Example}
\newcommand{\ZZ}{\mathbb{Z}}
\newcommand{\RR}{\mathbb{R}}
\newcommand{\CC}{\mathbb{C}}
\newcommand{\NN}{\mathbb{N}}
\newcommand{\EE}{\mathbb{E}}

\newcommand{\drm}{\ensuremath{\mathrm{d}}}
\newcommand{\abs}[1]{\left\vert#1\right\vert}
\newcommand{\sprod}[2]{\left\langle#1,#2 \right \rangle}
\newcommand{\norm}[1]{\left\Vert#1\right\Vert}
\newcommand{\Pro}{\ensuremath{P}}

\newcommand{\euler}{\mathrm{e}}

\DeclareMathOperator{\supp}{\operatorname{supp}}

\DeclareMathOperator{\im}{Im}

\renewcommand{\epsilon}{\varepsilon}
\renewcommand{\i}{\ensuremath{{\mathrm{i}}}}
\newcommand{\ii}{\ensuremath{{\mathrm{i}}}}
\newcommand{\BIGOP}[1]{\mathop{\mathchoice%
{\raise-0.22em\hbox{\huge $#1$}}%
{\raise-0.05em\hbox{\Large $#1$}}{\hbox{\large $#1$}}{#1}}}
\newcommand{\BIGboxplus}{\mathop{\mathchoice%
{\raise-0.35em\hbox{\huge $\boxplus$}}%
{\raise-0.15em\hbox{\Large $\boxplus$}}{\hbox{\large $\boxplus$}}{\boxplus}}}
\newcommand{\bigtimes}{\BIGOP{\times}}
\usepackage{microtype}
\begin{document}

\title{Localization criteria for Anderson models on locally finite graphs}
\author{Martin Tautenhahn}
\date{Technische Universit\"a{}t Chemnitz \\ Fakult\"a{}t f\"u{}r Mathematik \\ D-09107 Chemnitz \\ {\small URL: \url{http://www.tu-chemnitz.de/~mtau}}}

% \address{Martin Tautenhahn, Technische Universit\"a{}t Chemnitz, Fakult\"a{}t f\"u{}r Mathematik, D-09107 Chemnitz}
% \urladdr{www.tu-chemnitz.de/~mtau}
% \email{martin.tautenhahn@mathematik.tu-chemnitz.de}
%  % The correct dates will be entered by the editor

% \keywords{Anderson model, localization, fractional moment method, locally finite graphs}
% \subjclass{82B44, 60H25, 35J10}

%\PACS{71.23.An}

\maketitle
\begin{abstract}
We prove spectral and dynamical localization for Anderson models on locally finite graphs using the fractional moment method. Our theorems extend earlier results on localization for the Anderson model on $\ZZ^d$. We establish geometric assumptions for the underlying graph such that localization can be proven in the case of sufficiently large disorder. 
\\[2ex]
\end{abstract}
\section{Introduction}
A strong form of idealization in solid state physics is to deal with ideal crystals,
whose quantum mechanical properties can be modeled with periodic selfadjoint operators. Such operators always exhibit only absolute continuous spectrum. However, true materials will have distortions (e.\,g. dislocations, vacancies, presence of impurity atoms), and their modeling leads to the study of random Schr\"o{}dinger operators. In this paper we study spectral properties of certain discrete random Schr\"o{}dinger operators $H$. A motivation for studying spectral properties is their close relation (e.\,g. via the RAGE-Theorem) with dynamical properties of the wave packets $\psi_t (\cdot) = \euler^{-\i t H} \psi_0 (\cdot)$ governed by the time-dependent Schr\"o{}dinger equation.
\par
A fundamental result for random Schr\"o{}dinger operators is the physical phe\-nom\-e\-non of localization. There are different mathematical formulations for this. We discuss two of them; spectral localization, i.\,e. the almost sure absence of continuous spectrum, and dynamical localization, i.\,e. the wave packets $\psi_t (\cdot) = \euler^{-\i t H} \psi_0 (\cdot)$ stay trapped in a finite region of space for all time (almost surely).
\par
The prototype of a random Schr\"o{}dinger operator is the so-called Anderson model \cite{Anderson1958}. The Anderson model is given by the discrete random Schr\"o{}dinger operator $h_\omega = h_0 + \lambda v_\omega$ on $\ell^2 (\ZZ^d)$ where $h_0$ denotes the negative discrete Laplacian, representing the kinetic energy, $v_\omega$ is a random potential given by a collection of independent identically distributed (i.\,i.\,d.) random variables $\{\omega_k \}_{k \in \ZZ^d}$, and the parameter $\lambda \geq 0$ measures the strength of the disorder. In the higher-dimensional case there are two methods to prove spectral and dynamical localization, respectively. The \emph{multiscale analysis} introduced by Fr\"o{}hlich and Spencer in \cite{FroehlichS1983} and further developed, e.\,g., in \cite{FroehlichMSS1985,DreifusK1989}, and the \emph{fractional moment method} introduced by Aizenman and Molchanov in \cite{AizenmanM1993} and further developed, e.\,g., in \cite{Graf1994,Hundertmark2000,AizenmanFSH2001}. The first paper concerning dynamical localization for the Anderson model on $\ZZ^d$ was \cite{Aizenman1994}. A typical localization result in the case of large disorder (and under some mild regularity condition on the distribution of $\{\omega_k\}_{k \in \ZZ^d}$) according to the fractional moment method is the following statement: 
\begin{theorem}[\cite{AizenmanM1993}] \label{theorem:aizenman}
 Let $s \in (0,1)$ and assume that $\lambda > (2d C)^{1/s}$ where $C$ is a constant depending only on $s$ and on the distribution of $\{\omega_k\}_{k \in \ZZ^d}$. Then $h_\omega$ has almost surely only pure point spectrum with exponentially localized eigenfunctions.
\end{theorem}
Aizenman and Molchanov stated in \cite{AizenmanM1993} that this result also applies to random Schr\"o{}\-ding\-er operators on graphs with a uniform bound on the vertex degree. The quantity $2d$ in Theorem~\ref{theorem:aizenman} then has to be replaced by the uniform bound on the vertex degree. 
\par
In this paper we extend Theorem \ref{theorem:aizenman} to Anderson models on a certain class of locally finite graphs including all graphs with a uniform bound on the vertex degree and also certain graphs which have no uniform bound on the vertex degree. We prove spectral and dynamical localization in the case of sufficiently large disorder. We establish a geometric quantity which corresponds to the term $2d$ from Theorem~\ref{theorem:aizenman}. For a class of graphs which Hammersley \cite{Hammersley1957} called crystals, this quantity can be related to the so-called connective constant.
%
% %
% %%%%%%%%%%%%%%%%%%%%%%%%%%%%%%%%%%%%%%%%%%%%%%%%%%%%%%%%%%%%%%%%%%%%%%%%%%%%%%%%%%%
% %----------------------------------------------------------------------------------
% %         M O D E L    A N D    R E S U L T S
% %----------------------------------------------------------------------------------
% %%%%%%%%%%%%%%%%%%%%%%%%%%%%%%%%%%%%%%%%%%%%%%%%%%%%%%%%%%%%%%%%%%%%%%%%%%%%%%%%%%%
% %
%
\section{Model and results}\label{sec:model}
Let $G = (V,E)$ be an infinite, locally finite, connected graph without loops or multiple edges where $V = V (G)$ is the set of vertices and $E = E(G)$ is the set of edges. We use the notation $x \sim y$ to indicate that an edge connects the vertices $x$ and $y$. Moreover, we denote by $m(x)$ the valence at a vertex $x \in V$, that is, $m(x) := \# \{y \in V \colon y \sim x\}$ denotes the number of vertices connected by an edge to $x$. Recall that locally finite means $m(x) < \infty$ for all $x \in V$, while $m(x)$ is not necessarily uniformly bounded. As usual, we denote by $d(x,y)$ the graph distance of $x$ and $y$, that is, the length of the shortest $x-y$ path in $G$, where the length of a path is the number of its edges. Recall that $G$ is assumed to be connected, so that $d(x,y) \in \NN_0$
for all $x,y\in V$. Further, for $n \in \NN$ and $x \in V$, the symbol $c_x (n)$ denotes the number of self-avoiding walks with length $n$ starting in $x$, where we set $c_x (0) = 1$. A self-avoiding walk is a walk which never intersects itself. Let us moreover introduce two geometric assumptions which may hold or not.
\begin{assumption} \label{ass:spectral}
There exists $\alpha \in (0,1)$, such that for all $y \in V$
\begin{equation} \label{eq:spectral}
 \sum_{k \in V} \alpha^{d(k,y)} c_k(d(k,y)) < \infty .
\end{equation}
\end{assumption}
\begin{assumption} \label{ass:dyn}
There exists $\beta \in (0,1)$, such that for some $o \in V$, all $y \in V$ and any $p\geq0$
\begin{equation} \label{eq:dynamical}
\sum_{x \in V}
\sum_{k \in V} 
\lvert d(o,x) \rvert^p \Bigl(\beta^{d(x,k) + d(k,y)} c_{x}(d(x,k)) c_{k}(d(k,y))\Bigr)^{\frac{1}{2}} < \infty .
\end{equation}
\end{assumption}
\begin{definition}
If Assumption \ref{ass:spectral} is satisfied we set
\[
\alpha^* := \sup \bigl\{ \alpha \in (0,1) \colon \text{\eqref{eq:spectral} holds true for all $y \in V$}\bigr\} .
\] 
If Assumption \ref{ass:dyn} is satisfied we set 
\[
\beta^* := \sup \bigl\{ \beta \in (0,1) \colon \text{\eqref{eq:dynamical} holds true for some $o \in V$, all $y \in V$ and any $p > 0$} \bigr\}. 
\]
\end{definition}
The validity of Assumption \ref{ass:spectral} and \ref{ass:dyn} as well as quantitative estimates for the critical values $\alpha^*$ and $\beta^*$ are discussed in Remark \ref{remark:assumptions} at the end of this section.
\par
Let us introduce the Hilbert space $\ell^2 (V) := \{\psi : V \to \CC \mid \sum_{x \in V} |\psi (x)|^2 < \infty\}$ with inner product $\sprod{\psi}{\phi} := \sum_{x \in V} \overline{\psi(x)} \phi(x)$ and denote by $C_{\rm c} (V) \subset \ell^2 (V)$ the dense subset of functions $\psi : V \to \CC$ with finite support. On $\ell^2 (V)$, we define the discrete Laplace operator $\Delta_c : C_{\rm c} (V) \to \ell^2 (V)$, representing the kinetic energy term, by
\begin{equation*} 
 \bigl(\Delta_{\rm c} \psi \bigr)(x) := -m(x) \psi (x) + \sum_{y \sim x} \psi (y) .
\end{equation*}
Notice that $\Delta_{\rm c}$ is unbounded, iff there is no uniform bound on the vertex degree. However, it is known that $\Delta_{\rm c}$ is essentially selfadjoint, see e.\,g. \cite{Wojciechowski2007,Weber2010,Jorgensen2008}, or \cite{KellerL2011} for a proof in a more general framework. Therefore, $\Delta_{\rm c}$ has a unique selfadjoint extension, which we denote in the following by $\Delta : D(\Delta) \subset \ell^2 (V) \to \ell^2 (V)$.
The random potential term $V_\omega : \ell^2 (V) \to \ell^2 (V)$ is defined by
\[
 \bigl(V_\omega \psi\bigr) (x) := \omega_x \psi (x) .
\]
Here, we assume that $\{\omega_x\}_{x \in V}$ is given by a collection of independent identically distributed (i.\,i.\,d.) random variables, each distributed with the same density $\rho \in L^\infty(\RR) \cap L^1 (\RR)$. We assume that $\supp \rho = \{t \in \RR \colon \rho (t) \not = 0\}$ is a bounded set. The symbol $\mathbb{E}\{\cdot\}$ denotes the expectation with respect to the probability measure, i.\,e. $\mathbb{E} \{\cdot\} := \int_\Omega (\cdot) \prod_{k \in V} \rho(\omega_k) \drm \omega_k$ where $\Omega = \bigtimes_{k \in V} \RR$. For a set $\Gamma \subset V$, $\mathbb{E}_\Gamma\{\cdot \}$ denotes the expectation with respect to $\omega_k$, $k \in \Gamma$. That is, $\mathbb{E}_{\Gamma} \{\cdot\} := \int_{\Omega_\Gamma} (\cdot) \prod_{k \in \Gamma} \rho(\omega_k) \drm \omega_k$ where $\Omega_\Gamma := \bigtimes_{k \in \Gamma} \RR$. Since $\rho$ has finite support, $V_\omega$ is a bounded operator. We define the operator $H_\omega : D(\Delta) \to \ell^2 (V)$ for $\lambda > 0$ by
\begin{equation} \label{eq:hamiltonian}
 H_\omega := -\Delta + \lambda V_\omega .
\end{equation}
Since $\Delta$ is selfadjoint and $V_\omega$ is bounded and symmetric, it follows that $H_\omega$ is a selfadjoint operator. For the operator $H_\omega$ in Eq.~\eqref{eq:hamiltonian} and $z \in \CC \setminus \RR$ we denote the corresponding \emph{resolvent} by $G_\omega (z) := (H_\omega - z)^{-1}$. For the \emph{Green function}, which assigns to each  $(x,y) \in V \times V$
the corresponding matrix element of the resolvent, we use the notation
\begin{equation} \label{eq:greens}
G_\omega (z;x,y) := \sprod{\delta_x}{(H_\omega - z)^{-1}\delta_y}.
\end{equation}
Here, for $k \in V$, $\delta_k \in \ell^2 (V)$ denotes the Dirac function given by $\delta_k (k) = 1$ for $k \in V$ and $\delta_k (j) = 0$ for $j \in V \setminus \{k\}$. Notice that $\{\delta_k \colon k \in V\}$ is a complete orthonormal system for $\ell^2 (V)$.
\par
The quantity $\lambda / \Vert \rho \Vert_\infty^{-1}$ may be understood as a measure of the disorder present in the model. Our results are the following three theorems.
\begin{theorem} \label{theorem:result1}
 Let $s \in (0,1)$. Then we have for all $x,y \in V$ and all $z \in \CC \setminus \RR$
\[
 \EE \bigl\{ |G_\omega (z;x,y)|^s \bigr\} \leq 
 C' C^{d(x,y)} c_x(d(x,y)) ,
\]
where
\[
 C := \lambda^{-s} \lVert \rho \rVert_\infty^s \frac{2^s s^{-s}}{1-s} \quad \text{and} \quad C' := 2^{s+1} C .
\]

\end{theorem}
\begin{theorem} \label{theorem:result2}
 Let $s \in (0,1)$ and Assumption \ref{ass:spectral} be satisfied. Assume that 
\[
 \frac{\lambda}{\norm{\rho}_\infty} > \left(\frac{1}{\alpha^*} \cdot \frac{2^s s^{-s}}{1-s} \right)^{1/s} .
\]
Then $H_\omega$ exhibits almost surely only pure point spectrum, i.\,e. $\sigma_{\rm c} (H_\omega) = \emptyset$ almost surely.
\end{theorem}
\begin{theorem} \label{theorem:result3}
Let $s \in (0,1)$ and Assumption \ref{ass:dyn} be satisfied. Assume that 
\[
 \frac{\lambda}{\norm{\rho}_\infty} > \left(\frac{1}{\beta^*} \cdot \frac{2^s s^{-s}}{1-s} \right)^{1/s} .
\]
Then we have for some $o \in V$, any $p \geq 0$, $\psi \in \ell^2 (V)$ of compact support and $a,b \in \RR$ with $a<b$ that
\[
 \sup_{t \in \RR} \Bigl\Vert \lvert X_o \rvert^p \euler^{-\i t H_\omega} P_{(a,b)} \psi\Bigr\Vert < \infty \quad \text{almost surely} .
\]
\end{theorem}
Here, for an interval $I \subset \RR$, $P_I = P_I (H_\omega)$ denotes the spectral projection onto the Interval $I$ associated to the operator $H_\omega$. For $o \in V$ the operator $X_o : \ell^2 (V) \to \ell^2 (V)$ denotes the position operator given by $\bigl(X_o \psi\bigr)(x) = d(x,o) \psi(x)$.
\par
Theorem \ref{theorem:result1} concerns an estimate of an averaged fractional power of the Green function in terms of selfavoiding walks.  In the case of sufficiently large disorder, the constant $C$ is smaller than one. Thus, if the number of self-avoiding walks starting in $x$ do not increase ``too fast'' with their length, then Theorem \ref{theorem:result1} may be understood as an decay estimate for off-diagonal Green function elements. 
Theorem \ref{theorem:result1} is proven in Section \ref{sec:fmb}, while the boundedness of an averaged fractional power of the Green function is proven in Section \ref{sec:bounded}. The proofs are based on methods developed in \cite{AizenmanM1993,Graf1994,AizenmanFSH2001}. 
\par
Theorem \ref{theorem:result2} and Theorem \ref{theorem:result3} concern spectral and dynamical localization. They are proven in Section \ref{sec:spectral_loc}. While the conclusion from fractional moment bounds to localization is well developed for the lattice $\ZZ^d$, see e.\,g. \cite{AizenmanM1993,Graf1994,AizenmanFSH2001}, the proof of our results causes difficulties due to the fact that our Laplacian is unbounded. 
\begin{remark}
Notice that dynamical localization implies spectral localization by the RAGE-theorem, see e.\,g. \cite{Stolz2010}, but not vice versa as examples in \cite{RioJLS1996} show. From this fact and Theorem~\ref{theorem:result3} (on dynamical localization) it follows that under Assumption~\ref{ass:dyn} we have spectral localization in the large disorder regime. Under Assumption~\ref{ass:spectral}, however, we do not know whether dynamical localization holds true. Further, at present, we are not able to establish any relation between the conditions \ref{ass:spectral} and \ref{ass:dyn}. For this reason, we give a separate proof of spectral localization under Assumption~\ref{ass:spectral} in Section~\ref{sec:spectral_loc}.
\end{remark}
\begin{remark} \label{remark:assumptions}
In this remark we discuss the validity of Assumption \ref{ass:spectral} and \ref{ass:dyn}, and quantitative estimates of $\alpha^*$ and $\beta^*$. See Section \ref{sec:appendix} for more details and proofs. For $y \in V$ and $n \in \NN$ we denote by $S_y (n) = \{x \in V : d(x,y) = n\}$ the sphere of radius $n$ centered at $y$. For finite $\Gamma \subset V$, $\lvert \Gamma \rvert$ denotes the number of elements of $\Gamma$.
\begin{enumerate}[(i)]
 \item Assume that $G$ has uniformly bounded vertex degree, i.\,e. there is a $K \in \NN$ such that $m(x) \leq K$ for all $x \in V$. Then Assumptions \ref{ass:spectral} and \ref{ass:dyn} are fulfilled.
 \item There are also graphs with no uniform bound on the vertex degree which satisfy Assumption \ref{ass:spectral} and \ref{ass:dyn}. We give two examples.
The first one is a rooted tree where the number of offsprings $O(g)$ in generation $g \in \NN_0$ is given by $O(g) = \log_2 g$ if $\log_2 g \in \NN$ and $O(g) = 1$ else. The second example has vertex set $V = \ZZ^2$. Two vertices $x$ and $y$ are connected by an edge if their $\ell^1$-distance is one. Furthermore, for $n=3,4,5,\dots$, the vertex $(2^n,0)$ is connected to all vertices whose $\ell^1$ distance to $(2^n,0)$ equals $n$.
Obviously, these two graphs have no uniform bound on their vertex degree. However, both graphs obey Assumptions \ref{ass:spectral} and \ref{ass:dyn}.
\par
What matters that a graph $G = (V,E)$ satisfies Assumption \ref{ass:spectral} is that for each $y \in V$ there are $a , b > 0$ such that $c_y (n) \leq a \cdot b^n$ for all $n \in \NN$. If the constants $a$ and $b$ exist uniformly in $y \in V$ then the graph $G$ also satisfies Assumption \ref{ass:dyn}.
 \item Hammersley proved in \cite{Hammersley1957} for a class of graphs called \emph{crystals} that there exists a constant $\mu$, called the \emph{connective constant} of a graph, such that
\[
 \lim_{n \to \infty} c_x (n)^{\frac{1}{n}} = \mu \quad \text{for all $x \in V$}.
\]
Now assume that $G$ is a crystal. Then the following statements hold true:
\begin{enumerate}[(a)]
 \item Assume that for each vertex $y \in V$ there is a polynomial $p_y$ such that $\lvert S_y (n)\rvert$ $\leq p_y (n)$ for all $n \in \mathbb{N}$. Then Assumption \ref{ass:spectral} is satisfied and $\alpha^* = 1 / \mu$. 
 \item Assume there is a polynomial $p$, such that $\lvert S_y (n) \rvert \leq p(n)$ for all $n \in \mathbb{N}$ and $y \in V$. Then Assumption \ref{ass:dyn} is satisfied and $\beta^* = 1/\mu$.
 \item Assume that there are $a \in (0,\infty)$ and $b \in (1,\infty)$ such that $\lvert S_y (n) \rvert \leq ab^n$ for all $y \in V$ and $n \in \mathbb{N}$. Then Assumptions \ref{ass:spectral} and \ref{ass:dyn} are satisfied and $\alpha^* , \beta^* \geq 1/(b \mu)$. 
\end{enumerate}
In particular, for the lattice $\ZZ^d$ with standard edges we can replace $1/\alpha^*$ and $1/\beta^*$ in Theorem \ref{theorem:result2} and \ref{theorem:result3} by $\mu$. Notice that $\mu$ is typically smaller than $2d-1$, see \cite{Hammersley1963}. Hence, we have weakened the disorder assumption of Theorem \ref{theorem:aizenman}.
\end{enumerate}
\end{remark}
\begin{remark}
 Theorem \ref{theorem:result2} gives almost surely absence of continuous spectrum for the operator $H_\omega = -\Delta + \lambda V_\omega$ in the case of sufficiently large disorder whenever the graph $G$ satisfies Assumption \ref{ass:spectral}. However, depending on the graph $G$, even the graph Laplacian itself may have no continuous spectrum. Keller proves in \cite{Keller2010} the following theorem.
\begin{theorem}[\cite{Keller2010}] \label{theorem:keller}
 Let $G = (V,E)$ be infinite and $a_\infty > 0$. Then $\sigma_{\rm ess} (\Delta) = \emptyset$ if and only if $m_\infty = \infty$.
\end{theorem}
Here $m_\infty = \lim_{K \to \infty} m_K$ where $m_K = \inf\{m(v) \mid v \in V \setminus K\}$ for finite $K \subset V$, and $a_\infty$ is the Cheeger constant at infinity introduced in \cite{Fujiwara1996}. For the precise meaning of the limit $K \to \infty$ we refer to \cite{Keller2010}. Notice that the continuous spectrum of an operator is always contained in the essential spectrum.
\par
The assumptions of Theorem~\ref{theorem:keller} and Assumption~\ref{ass:spectral} are in some sense contrary. Assumption~\ref{ass:spectral} may be interpreted that the graph grows slowly, this justifies to think of the Laplacian (at least on finite sets) as a perturbation of $V_\omega$. The assumptions of Theorem~\ref{theorem:keller} may be interpreted that the graph grows rapidly. Notice that the examples in Remark~\ref{remark:assumptions} satisfy Assumption~\ref{ass:spectral}, but do not satisfy $m_\infty = \infty$. It is an open question if there are graphs satisfying Assumption~\ref{ass:spectral}, $a_\infty > 0$ and $m_\infty = \infty$.
\end{remark}
%
%
% %
% %%%%%%%%%%%%%%%%%%%%%%%%%%%%%%%%%%%%%%%%%%%%%%%%%%%%%%%%%%%%%%%%%%%%%%%%%%%%%%%%%%%
% %----------------------------------------------------------------------------------
% %         B O U N D E D N E S S
% %----------------------------------------------------------------------------------
% %%%%%%%%%%%%%%%%%%%%%%%%%%%%%%%%%%%%%%%%%%%%%%%%%%%%%%%%%%%%%%%%%%%%%%%%%%%%%%%%%%%
% %
% 
\section{Boundedness of Green's function} \label{sec:bounded}
To show the boundedness of averaged fractional powers of the Green function we use the method proposed in \cite{Graf1994} for the case $V = \ZZ^d$. The proof is independent of the underlying geometry and applies directly to Anderson models on locally finite graphs.
\par
Let $\Gamma \subset V$. All through this paper, we will assume that either $\Gamma$ is finite or $V \setminus \Gamma$ is finite (or the empty set). For $k \in \Gamma$ we define the Dirac function $\delta_k \in \ell^2 (\Gamma)$ by $\delta_k (k) = 1$ and $\delta_k (j) = 0$ for $j \in \Gamma \setminus \{k\}$. Let the operators $P_x : \ell^2 (\Gamma) \to \ell^2 (\Gamma)$ and $\Pro_{\Gamma} : \ell^2 (V) \to \ell^2 (\Gamma)$ be given by
\begin{equation} \label{gl:plambda}
P_x \psi = \psi(x) \delta_x \quad \text{and} \quad \Pro_{\Gamma} \psi := \sum_{k \in \Gamma} \psi (k) \delta_k .
\end{equation}
Note that the adjoint $(\Pro_{\Gamma})^* : \ell^2 (\Gamma) \to \ell^2 (V)$ is given by $(\Pro_{\Gamma})^* \phi = \sum_{k \in \Gamma} \phi (k) \delta_k$. On $\ell^2 (\Gamma)$, the restricted operators $\Delta_\Gamma$, $V_\Gamma$ and $H_\Gamma$ are formally given by
\[
\Delta_\Gamma := \Pro_\Gamma \Delta \Pro_\Gamma^\ast, \quad V_\Gamma := \Pro_\Gamma V_\omega \Pro_\Gamma^\ast \quad \text{and} \quad H_\Gamma := \Pro_\Gamma H_\omega \Pro_\Gamma^\ast = -\Delta_\Gamma + \lambda V_\Gamma .
\]
We say formally since we did not say anything about their domain. If $\Gamma$ is finite, the operators $\Delta_\Gamma$, $V_\Gamma$ and $H_\Gamma$ are bounded and thus selfadjoint on the domain $\ell^2 (\Gamma)$. Otherwise, if $V \setminus \Gamma$ is finite, $\Delta_\Gamma$ may be understood as a bounded perturbation of the Laplace operator of the induced subgraph $G[\Gamma]$, and is thus essentially selfadjoint on the domain $C_{\rm c} (\Gamma)$. Therefore, $\Delta_\Gamma : C_{\rm c}(\Gamma) \to \ell^2 (\Gamma)$ has a unique selfadjoint extension, which we again denote by $\Delta_\Gamma : D (\Delta_\Gamma) \to \ell^2 (\Gamma)$. Since $V_\Gamma$ is bounded, $H_\Gamma : D(\Delta_\Gamma) \to \ell^2 (\Gamma)$ is selfadjoint.
%\par
We denote the resolvent of $H_\Gamma$ by $G_\Gamma (z) := (H_\Gamma - z)^{-1}$ and the Green function by $G_\Gamma (z;x,y) := \bigl\langle \delta_x, G_\Gamma (z) \delta_y \bigr\rangle$ for $z \in \CC \setminus \RR$ and $x,y \in \Gamma$. 
\par
To prove the boundedness of an averaged fractional power of the Green function we use the fact that for $\Gamma \subset V$, all $x,y \in \Gamma$ and $z \in \CC \setminus \RR$ we have the representation
\begin{equation} \label{eq:green_rep}
G_\Gamma (z;x,x) = \frac{\lambda^{-1}}{\omega_x - \beta} \quad \text{and} \quad
\lvert G_\Gamma (z;x,y) \rvert \leq \frac{\lambda^{-1}}{\lvert \omega_x - \gamma \rvert} + \frac{\lambda^{-1}}{\lvert \omega_y - \delta \rvert},
\end{equation}
where $\beta$ and $\gamma$ do not depend on $\omega_x$ and $\delta$ is independent of $\omega_y$. This was shown in \cite{Graf1994} in the case $V = \ZZ^d$ and applies directly to our setting. A second fact is standard spectral averaging. More precisely, let $g:\RR \to \RR$ non-negative with $g \in L^{\infty}(\RR) \cap L^1(\RR)$ and $s \in (0,1)$. Then we have for all $\beta \in \CC$
\begin{equation} \label{lemma:average}
\int_\RR \abs{\xi - \beta}^{-s} g(\xi) \drm \xi \leq \norm{g}_\infty^s \norm{g}_{L^1}^{1-s} \frac{2^s s^{-s}}{1-s} ,
\end{equation}
see e.\,g. \cite{Graf1994} for a proof.
Equation \eqref{eq:green_rep} and Eq. \eqref{lemma:average} imply the following lemma on the boundednes of an averaged fractional power of the Green function.
\begin{lemma} \label{lemma:finitness2}
 Let $s \in (0,1)$ and $\Gamma \subset V$. Then we have for all $x,y \in \Gamma$ with $x \not = y$ and all $z \in \CC \setminus \RR$ the estimates
\begin{equation*}
 \mathbb{E}_{\{x\}} \Bigl\{ \bigl|G_\Gamma (z;x,x)\bigr|^s \Bigr\} \leq \lambda^{-s} \norm{\rho}_\infty^s \frac{2^s s^{-s}}{1-s} = C 
\end{equation*}
and
\begin{equation*}
 \mathbb{E}_{\{x,y\}} \Bigl\{ \bigl|G_\Gamma (z;x,y)\bigr|^s \Bigr\} \leq \lambda^{-s} \norm{\rho}_\infty^s 2^{s+1} \frac{2^s s^{-s}}{1-s} = C' . 
\end{equation*}
\end{lemma}
% 
% 
% %
% %%%%%%%%%%%%%%%%%%%%%%%%%%%%%%%%%%%%%%%%%%%%%%%%%%%%%%%%%%%%%%%%%%%%%%%%%%%%%%%%%%%
% %----------------------------------------------------------------------------------
% %         Green's function Estimates
% %----------------------------------------------------------------------------------
% %%%%%%%%%%%%%%%%%%%%%%%%%%%%%%%%%%%%%%%%%%%%%%%%%%%%%%%%%%%%%%%%%%%%%%%%%%%%%%%%%%%
% %
%  
\section{Moment bounds in terms of self-avoiding walks; proof of Theorem~\ref{theorem:result1}} \label{sec:fmb}
In this section we consider so called ``depleted'' Hamiltonians. Such Hamiltonians are obtained by setting to zero the operators ``hopping terms'' along a collection of bonds. More precisely, let $\Gamma \subset V$ be an arbitrary set and $\Lambda \subset \Gamma$ be a finite set. We define the depleted Laplace operator $\Delta_\Gamma^\Lambda : D(\Delta_\Gamma) \to \ell^2 (\Gamma)$ by
\begin{equation*} \label{eq:de1}
 \sprod{\delta_x}{\Delta_\Gamma^\Lambda \delta_y} := 
\begin{cases}
  0 & \text{if $x \in \Lambda$, $y \in \Gamma \setminus \Lambda$ or $y \in \Lambda$, $x \in \Gamma \setminus \Lambda$}, \\
  \sprod{\delta_x}{\Delta_\Gamma \delta_y} & \text{else}.
\end{cases} 
\end{equation*}
In other words, the hopping terms which connect $\Lambda$ with $\Gamma \setminus \Lambda$ or vice versa are deleted. The depleted Hamiltonian $H_\Gamma^\Lambda : D(\Delta_\Gamma) \to \ell^2 (\Gamma)$ is then defined by
\[
 H_\Gamma^\Lambda := -\Delta_\Gamma^\Lambda + V_\Gamma .
\]
Let further $T_\Gamma^\Lambda := \Delta_\Gamma - \Delta_\Gamma^\Lambda$ be the difference between the ``full'' Laplace operator and the depleted Laplace operator. Since $\Lambda$ is assumed to be finite, $T_\Gamma^\Lambda$ is a bounded operator and thus $\Delta_\Gamma^\Lambda$ and $H_\Gamma^\Lambda$ are selfadjoint operators.
Analogously to Eq. \eqref{eq:greens} we use the notation $G_\Gamma^\Lambda (z) = (H_\Gamma^\Lambda - z)^{-1}$ and $G_\Gamma^\Lambda (z;x,y) = \bigl \langle \delta_x^\Gamma, G_\Gamma^\Lambda(z) \delta_y^\Gamma \bigr \rangle$. The second resolvent identity yields for arbitrary sets $\Lambda \subset \Gamma \subset V$
\begin{align} 
 G_\Gamma (z) 	& = G_\Gamma^\Lambda (z) + G_\Gamma (z) T_\Gamma^\Lambda G_\Gamma^\Lambda (z) .
 \label{eq:resolvent} 
\end{align}
Notice that $G_\Gamma^\Lambda (z;x,y) = G_{\Gamma \setminus \Lambda} (z;x,y)$ for all $x,y \in \Gamma \setminus \Lambda$ and that $G_\Gamma^\Lambda (z;x,y) = 0$ if $x \in \Lambda$ and $y \not \in \Lambda$ or vice versa, since $H_\Gamma^\Lambda$ is block-diagonal with the two blocks $\Lambda$ and $\Gamma \setminus \Lambda$.
\begin{lemma}
 Let $\Gamma \subset V$ and $s \in (0,1)$. Then we have for all $x,y \in \Gamma$ with $x \not = y$
\begin{equation} \label{eq:iteration}
 \EE_{\{x\}} \bigl\{|G_\Gamma (z;x,y)|^s \bigr\} \leq \lambda^s \lVert \rho \rVert_\infty^s \frac{2^s s^{-s}}{1-s}
\sum_{\{k\in \Gamma : k \sim x\}} |G_{\Gamma \setminus \{x\}} (z;k,y)|^s .
\end{equation}
\end{lemma}
\begin{proof}
 By assumption $\{x,y\} \subset \Gamma$. Starting point of the proof is Eq. \eqref{eq:resolvent} with the choice $\Lambda := \{x\}$. Taking the matrix element $(x,y)$ gives
\begin{equation*}
 G_\Gamma (z;x,y) = G_\Gamma^\Lambda (z;x,y) + \bigl\langle \delta_x ,  G_\Gamma (z) T_\Gamma^\Lambda G_\Gamma^\Lambda (z) \delta_y \bigr\rangle .
\end{equation*}
The first summand on the right hand side of the above identity is zero, since $x \in \Lambda$ and $y \not \in \Lambda$. For the second term we calculate
\begin{align} \label{eq:sawrepresentation}
 G_\Gamma (z;x,y) &= \bigl\langle \delta_x ,  G_\Gamma (z) T_\Gamma^\Lambda G_\Gamma^\Lambda (z) \delta_y \bigr\rangle 
= G_\Gamma (z;x,x) \sum_{\{k \in \Gamma : k \sim x \}} G_{\Gamma \setminus \Lambda} (z;k,y) .
\end{align}
Since $G_{\Gamma \setminus \Lambda} (z)$ is independent of $\omega_x$, we obtain the statement of the lemma by taking absolute value to the power of $s$, averaging with respect to $\omega_x$ and using Lemma \ref{lemma:finitness2}.
\end{proof}
\begin{remark}
If one iterates Eq.~\eqref{eq:sawrepresentation} exactly $d(x,y)$ times, starting with $\Gamma = V$, one obtains a representation formula for the Green function in terms of selfavoiding walks.
\par
In \cite{Hundertmark2008} Hundertmark also establishes in the case $V = \ZZ^d$ a representation for the finite volume Green function in terms of selfavoiding walks. However, the proof of this representation formula uses the fact that the Laplacian is bounded, which is true on the lattice $\ZZ^d$ but not in our setting.
\end{remark}
\begin{proof}[Proof of Theorem \ref{theorem:result1}]
If $x = y$, we obtain the statement of the theorem by Lemma \ref{lemma:finitness2}. Assume $x \not = y$ and set $l = d(x,y)$. In order to estimate $\EE \{|G_\omega (z;x,y)|^s\}$ we iterate Eq. \eqref{eq:iteration} exactly $l$ times, starting with $\Gamma = V$, and obtain
\[
 \EE \bigl\{ |G_\omega (z;x,y)|^s \bigr\} \leq \left( \lambda^s \lVert \rho \rVert_\infty^s \frac{2^s s^{-s}}{1-s} \right)^{l} \sum_{w \in {\rm SAW}(x,l)} 
\EE \Bigl\{\bigl|G_{V \setminus \Lambda(w)} \bigl(z;w(l) , y\bigr)\bigr|^s \Bigr\} .
\]
Here, for $x \in V$ and $n \in \NN$, ${\rm SAW}(x,n)$ denotes the set of all self-avoiding walks in $G$ starting in $x$ with length $n$. For $w\in {\rm SAW}$, $w(i)$ denotes the vertices visited at step $i$, and for $w \in {\rm SAW}(x,n)$ we denote $\Lambda (w) = \{w(0),\dots,w(n-1)\}$.
 By Lemma \ref{lemma:finitness2}, all summands are bounded by $C'$, which gives the desired statement.
\end{proof}
For the proof of spectral and dynamical localization in the following sections, we use that the second moment of Green's function can be bounded in terms of its fractional moments, see Lemma \ref{lemma:s=2} below. This fact was established in \cite{Graf1994} in the case $V = \ZZ^d$.
Notice that our proof of Lemma \ref{lemma:s=2} is a slightly modified version of Graf's proof, allowing for densities $\rho$ with unbounded support. However, we require $\supp \rho$ bounded for different reasons, namely to settle the question of selfadjointness.
\begin{lemma} \label{lemma:s=2}
 Let $s \in (0,1)$ and $C,C'$ be the constant from Theorem \ref{theorem:result1}. Then we have for all $x,y \in V$ and all $z \in \CC \setminus \RR$
\[
 \abs{\im z} \EE \bigl\{ |G_\omega (z;x,y)|^2 \bigr\} \leq  \max\{1,\pi\norm{\rho}_\infty\} C' C^{d(x,y)} c_x(d(x,y)) .
\]
\end{lemma}
For the proof of Lemma \ref{lemma:s=2} some preparatory estimates are required.
\begin{lemma} \label{lemma:sto2}
 Let $H$ be a selfadjoint operator on $\ell^2 (V)$, $s \in [0,2]$, $G(z) = (H-z)^{-1}$ and $G(z;x,y) = \langle \delta_x , G(z) \delta_y\rangle$. Then we have for all $x,y \in V$ and all $z \in \CC \setminus \RR$ the bounds
\begin{align*}
 \abs{\im z} \abs{G (z;x,y)}^2 &\leq \abs{\im G(z;x,x)} \frac{\abs{G(z;x,y)}^s}{\abs{G(z;x,x)}^s}
\intertext{and}
\abs{\im z} \frac{\abs{G(z;x,y)}^2}{\abs{G(z;x,x)}^2} &\leq \abs{\im G(z;x,x)^{-1}} \frac{\abs{G(z;x,y)}^s}{\abs{G(z;x,x)}^s} .
\end{align*}
\end{lemma}
For a proof of Lemma~\ref{lemma:sto2} we refer to \cite{Graf1994,AizenmanG1998}. Notice that the second statement of the lemma follows from the first statement of the lemma by using $\lvert \im z \rvert = \lvert \im z^{-1} \rvert \lvert z \rvert^2$ for $z \in \CC$ with $z \not = 0$.
\begin{proof}[Proof of Lemma \ref{lemma:s=2}]
 We start with two estimates for $\lvert \im z \rvert \EE_{\{x\}}\{|G_\omega (z;x,y)|^2\}$. We use the first estimate of Lemma \ref{lemma:sto2}, average with respect to $\omega_x$, estimate a part of the integrand by its maximum and obtain a \textit{first estimate}
\begin{equation*}
 \abs{\im z} \EE_{\{x\}} \bigl\{ \abs{G_\omega (z;x,y)}^2 \bigr\} \leq 
A(\omega) \EE_{\{x\}}\bigl\{\abs{G_\omega (z;x,y)}^s\bigr\} ,
\end{equation*}
where 
\[
 A(\omega) := \max_{\omega_x \in \RR} \left\lbrace \frac{\abs{\im G_\omega(z;x,x)}}{\abs{G_\omega (z;x,x)}^s} \right\rbrace .
\]
Now we consider for $x \in V$ and $\kappa \in \RR$ the operator $H_{\omega,\kappa} = H_\omega + \kappa P_x$. We denote $G_{\omega , \kappa} (z) = (H_{\omega , \kappa} - z)^{-1}$ and $G_{\omega , \kappa}(z;x,y) = \langle \delta_x , G_{\omega , \kappa}(z) \delta_y \rangle$ and obtain from the second resolvent identity
\[
 G_{\omega , \kappa} (z;x,y) = \frac{1}{\kappa + G_{\omega}(z;x,x)^{-1}} \frac{G_\omega (z;x,y)}{G_\omega (z;x,x)} .
\]
Due to the second estimate of Lemma \ref{lemma:sto2} we obtain
\[
 \abs{\im z} |G_{\omega , \kappa} (z;x,y)|^2 \leq \frac{|\im G_\omega (z;x,x)^{-1}|}{|\kappa + G_{\omega}(z;x,x)^{-1}|^2} \frac{|G_\omega (z;x,y)|^s}{|G_\omega (z;x,x)|^s} .
\]
We multiply both sides with $\rho(\omega_x + \kappa)\rho(\omega_x)$, integrate with respect to $\kappa$ and $\omega_x$ and obtain a \textit{second estimate}
\[
  \abs{\im z} \EE_{\{x\}} \bigl\{ \abs{G_\omega (z;x,y)}^2 \bigr\} \leq 
B (\omega)
\EE_{\{x\}}\bigl\{\abs{G_\omega (z;x,y)}^s\bigr\}
\]
with 
\[
 B (\omega) := \max_{\omega_x \in \RR} \left\lbrace \int_\RR \frac{|\im G_\omega (z;x,x)^{-1}|}{|\kappa + G_\omega (z;x,x)^{-1}|^2} \frac{\rho(\omega_x + \kappa)}{|G_\omega (z;x,x)|^s} \drm \kappa \right\rbrace .
\]
Here we have used that $\int_\RR |G_{\omega , \kappa}(z;x,y)|^2 \rho(\omega_x + \kappa) \drm \kappa = \int_\RR |G_{\omega}(z;x,y)|^2 \rho(\omega_x) \drm \omega_x$, which is easy to see by the substitution $\omega_x + \kappa = t$. Combining our two estimates we obtain
\[
  \abs{\im z} \EE_{\{x\}} \bigl\{ \abs{G_\omega (z;x,y)}^2 \bigr\} \leq  \min \bigl\{A(\omega) , B(\omega)\bigr\} \cdot \EE_{\{x\}}\bigl\{\abs{G_\omega (z;x,y)}^s\bigr\}.
\]
Let now $\Omega_A := \{\omega \in \Omega \colon \lvert G_\omega(z;x,x) \rvert \leq 1\}$ and $\Omega_B := \{\omega \in \Omega \colon \abs{G_\omega (z;x,x) } \geq 1\}$. Then we have for $\omega \in \Omega_A$ the estimate
\[
 A(\omega) \leq \max_{\omega_x \in \RR} \left\lbrace |G_\omega (z;x,x)|^{1-s} \right\rbrace  \leq 1 .
\]
Otherwise, for $\omega \in \Omega_B$ we have, since $\int_\RR |\im b|/|t + b|^2 \drm t = \pi$ for $\im b \not = 0$,
\begin{align*}
 B(\omega) & \leq \max_{\omega_x \in \RR} \left\lbrace \norm{\rho}_\infty |G_\omega (z;x,x)|^{-s} \int_\RR \frac{|\im G_\omega (z;x,x)^{-1}|}{|\kappa + G_\omega (z;x,x)^{-1}|^2} \drm \kappa \right\rbrace 
 \leq \norm{\rho}_\infty \pi .
\end{align*}
Thus we have for all $x,y \in V$ (and all $\omega \in \Omega$)
\[
 \abs{\im z} \EE_{\{x\}} \bigl\{ \abs{G_\omega (z;x,y)}^2 \bigr\} \leq  \max \bigl\{1 , \norm{\rho}_\infty \pi \bigr\} \cdot \EE_{\{x\}}\bigl\{\abs{G_\omega (z;x,y)}^s\bigr\}.
\]
If we take full expectation, the desired result follows from Theorem \ref{theorem:result1}.
\end{proof}
%
%
% 
% 
%%%%%%%%%%%%%%%%%%%%%%%%%%%%%%%%%%%%%%%%%%%%%%%%%%%%%%%%%%%%%%%
% S P E C T R A L     L O C A L I Z A T I O N
%%%%%%%%%%%%%%%%%%%%%%%%%%%%%%%%%%%%%%%%%%%%%%%%%%%%%%%%%%%%%%%
%
% 
%
\section{Localization; proof of Theorem \ref{theorem:result2} and \ref{theorem:result3}} \label{sec:spectral_loc}
\subsection{Spectral localization; proof of Theorem \ref{theorem:result2}}
In this subsection we prove Theorem \ref{theorem:result2}, i.\,e. that $H_\omega$ exhibits almost surely only pure point spectrum. To this end we will use a local version of \cite[Eq. (2)]{Graf1994} established in Lemma~\ref{lemma:graf}. While the non-local version (i.\,e. \cite[Eq. (2)]{Graf1994}) can be proven in an elegant way using Parseval's identity, the proof of the local version is not a straightforward generalization, i. e. the idea with Parseval's identity cannot be applied. To our best knowledge there is no proof for the local version in the literature.
For the proof of Lemma \ref{lemma:graf} we shall need
\begin{lemma} \label{lemma:approx_ident}
 Let $f :\RR \to \CC$ be a bounded and piecewise continuous function. Then we have for all $a \in \RR$
\[
 I = \lim_{\epsilon \searrow 0} \frac{\epsilon}{\pi} \int_\RR \frac{f(E)}{(a - E)^2 + \epsilon^2} \drm E = \frac{1}{2} \left(  f (a-) + f (a+)\right)  ,
\]
where $f (a-) = \lim_{\lambda \nearrow a} f(\lambda)$ and $f (a+) = \lim_{\lambda \searrow a} f(\lambda)$.
\end{lemma}
Representation theorems of this type are well known in different settings. For example, if $f \in L^1 (\RR)$ such a statement (with $f(a)$ on the right hand side) has been proven for almost all $a \in \RR$, e.\,g., in \cite{Bear1979}.
\begin{proof}[Proof of Lemma \ref{lemma:approx_ident}]
Fix $a \in \RR$ and assume that $f$ is not continuous in $a$. The case where $f$ is continuous in $a$ is similar but easier. Let $\beta > 0$ be small enough, such that there is no further  discontinuity point in $[a-\beta , a+\beta]$.
We have the identity
\begin{align*}
J &= I - \frac{1}{2} \bigl( f (a-) + f(a+) \bigr ) \\
  &= \lim_{\epsilon \searrow 0} \frac{\epsilon}{\pi} \left( 
\int_{-\infty}^{a} \frac{f(E)- f (a-)}{(a - E)^2 + \epsilon^2} \drm E  +  
\int_{a}^{\infty} \frac{f(E)- f (a+)}{(a - E)^2 + \epsilon^2} \drm E 
\right) .
\end{align*}
For $\delta \in (0,\beta)$ we decompose the integration according to
\begin{multline*}
J =  \lim_{\epsilon \searrow 0} \frac{\epsilon}{\pi}  \left(   \left( 
\int_{-\infty}^{a-\delta} + \int_{a-\delta}^{a} \right) \frac{f(E)- f (a-)}{(a - E)^2 + \epsilon^2} \drm E \right.  \\ + \left. \left( 
\int_{a}^{a+\delta} + \int_{a+\delta}^{\infty} \right) \frac{f(E)- f (a+)}{(a - E)^2 + \epsilon^2} \drm E    \right) .
\end{multline*}
We denote by $I_1,I_2,I_3,I_4$ the four above integrals, in the order in which they occur. 
Since $f$ is bounded, we have for arbitrary $\delta \in (0,\beta)$
\begin{align*}
 \lvert I_1 \rvert &\leq \lVert f \rVert_\infty \int_{-\infty}^{a-\delta} \frac{1}{(a - E)^2 + \epsilon^2} \drm E + \lvert f(a-)\rvert \int_{-\infty}^{a-\delta} \frac{1}{(a - E)^2 + \epsilon^2} \drm E \\[1ex]
& = \left(\lVert f \rVert_\infty +  \lvert f(a-)\rvert \right) \frac{1}{\epsilon} \left( \frac{\pi}{2} - \arctan \left(\frac{\delta}{\epsilon} \right) \right) ,
\end{align*}
which gives $\lim_{\epsilon \searrow 0} \frac{\epsilon}{\pi} |I_1| = 0$, and similarly $\lim_{\epsilon \searrow 0} \frac{\epsilon}{\pi} |I_4| = 0$ for all $\delta \in (0,\beta)$.
For $I_3$ we estimate
\begin{align*}
 \lvert I_3  \rvert &= \left | \int_{a}^{a+\delta} \frac{f(E) - f (a+)}{(a - E)^2 + \epsilon^2} \drm E \right | \leq \sup_{E\in (a,a+\delta]} \bigr\{ |f(E) - f (a+)| \bigl\} \frac{1}{\epsilon} \arctan \left( \frac{\delta}{\epsilon} \right) .
\end{align*}
An analogue estimate holds for $I_2$. Summarizing, we have for arbitrary $\delta \in (0,\beta)$ that
\begin{align*}
 \lvert J \rvert &\leq \sup_{E\in (a,a+\delta]} \bigr\{ |f(E) - f (a+)| \bigl\} + \sup_{E\in [a-\delta,a)} \bigr\{ |f(E) - f (a-)| \bigl\}.
\end{align*}
Since $\delta \in (0,\beta)$ was arbitrary and $f (a+)$ ($f(a-)$) is the right-hand limit (left-hand limit) of $f$ in $a$, we obtain $|J| = 0$, which proves the statement of the lemma.
\end{proof}
\begin{lemma} \label{lemma:graf}
 Let $H$ be a selfadjoint operator on $\ell^2 (V)$, $P : \ell^2 (V) \to \ell^2 (V)$ be an orthogonal projection, $a,b \in \RR$ with $a<b$ and $P_{(a,b)} : \ell^2 (V) \to \ell^2 (V)$ be the spectral projection onto the interval $[a,b]$ associated to the operator $H$. Then,
\begin{equation} \label{eq:graf}
 \lim_{\epsilon \searrow 0} 2 \epsilon \int_0^\infty \euler^{-2\epsilon s} \norm{P \euler^{-\i H s} P_{(a,b)} \psi}^2 \drm s
 \leq \liminf_{\epsilon \searrow 0} \frac{\epsilon}{\pi}\int_a^b \norm{P (H-E-\i \epsilon)^{-1} \psi}^2 \drm E .
\end{equation}
\end{lemma}
\begin{proof}
 We denote the left hand side of Ineq. \eqref{eq:graf} by $L$ and the right hand side by $R$, respectively. For the right hand side $R$ we calculate by using the monotone convergence theorem and Fatou's lemma
 \begin{equation*}
 R \geq \sum_{k \in V} \liminf_{\epsilon \searrow 0}  \frac{\epsilon}{\pi}  \int_a^b  \sprod{P (H - E - \i \epsilon)^{-1} \psi}{\delta_k} \sprod{\delta_k}{P (H - E - \i \epsilon)^{-1} \psi} \drm E .
\end{equation*}
By the spectral theorem, polarization and Fubini's theorem we have
\begin{align*}
 R &\geq \sum_{k \in V} \liminf_{\epsilon \searrow 0}  \frac{\epsilon}{\pi}  \int_a^b \left(  \int_\RR \int_\RR \frac{\drm \langle P E_\lambda \psi , \delta_k \rangle \drm \langle \delta_k , P E_\mu \psi \rangle}{(\lambda - E + \i \epsilon)(\mu - E - \i \epsilon)}  \right) \drm E \\[1ex]
& =\sum_{k \in V} \liminf_{\epsilon \searrow 0}      \int_\RR \int_\RR f_\epsilon (\lambda , \mu) \drm \langle P E_\lambda \psi , \delta_k \rangle \drm \langle \delta_k , P E_\mu \psi \rangle ,
\end{align*}
where
\[
 f_\epsilon (\lambda , \mu) = \frac{\epsilon}{\pi}  \int_a^b \frac{\drm E }{(\lambda - E + \i \epsilon)(\mu - E - \i \epsilon)} .
\]
From Lemma \ref{lemma:approx_ident} we infer that $\lim_{\epsilon \searrow 0} f_\epsilon (\lambda , \lambda) = 1/2 (\chi_{[a,b]} (\lambda) + \chi_{(a,b)} (\lambda))$ for all $\lambda \in \RR$. For $\lambda \not = \mu$ one calculates that $\lim_{\epsilon \searrow 0} f_\epsilon (\lambda , \mu) = 0$. We also find that for all $\epsilon > 0$ and $\lambda , \mu \in \RR$
\begin{align*}
 \lvert f_\epsilon (\lambda, \mu) \rvert \leq \frac{\epsilon}{\pi}  \int_a^b \frac{\drm E }{\lvert (\lambda - E + \i \epsilon)(\mu - E - \i \epsilon) \rvert} \leq 2 .
\end{align*}
Recall that the limes inferior is the limit along a subsequence. By polarization and the dominated convergence theorem we obtain 
\begin{equation*}
 R \geq  \sum_{k \in V}  \int_\RR \int_\RR \delta_{\lambda \mu} \frac{1}{2} \left( \chi_{[a,b]} (\lambda) + \chi_{(a,b)} (\lambda) \right) \drm \langle P E_\lambda \psi , \delta_k \rangle \drm \langle \delta_k , P E_\mu \psi \rangle .
\end{equation*}
Here $\delta_{\lambda \mu}$ denotes the Kronecker function given by $\delta_{\lambda \lambda} = 1$ and $\delta_{\lambda \mu} = 0$ for $\lambda \not = \mu$. Since the both measures are complex conjugate to each other and the integrand is zero for $\lambda \not = \mu$, we obtain (using polar decomposition of the measures, see e.\,g. \cite{Rudin1987}),
\begin{equation*}
 R \geq \sum_{k \in V} \int_\RR \int_\RR \delta_{\lambda \mu} \chi_{(a,b)} (\lambda) \,\, \drm \langle P E_\lambda \psi , \delta_k \rangle \,\, \drm \langle \delta_k , P E_\mu \psi \rangle .
\end{equation*}
Similarly we calculate for the left hand side $L$ using the monotone convergence theorem and Fatou's lemma
\begin{align*}
 L &\leq \sum_{k \in V} \limsup_{\epsilon \searrow 0} 2 \epsilon \int_0^\infty \euler^{-2\epsilon s}  \Bigl| \sprod{P\euler^{-\i Hs} P_{(a,b)}\psi}{\delta_k}\Bigr|^2 \drm s .
\end{align*}
The spectral theorem, polarization and Fubini's theorem gives
\begin{align*}
L &\leq \sum_{k \in V} \limsup_{\epsilon \searrow 0} 2 \epsilon \int_0^\infty \left( \int_\RR \int_\RR \frac{\chi_{(a,b)}(\lambda) \chi_{(a,b)} (\mu)}{\euler^{(2\epsilon-\i\lambda+\i\mu) s}}  \, \drm \langle P E_\lambda  \psi , \delta_k \rangle \, \drm \langle \delta_k , P E_\mu  \psi \rangle \right) \drm s  \\[1ex]
&= \sum_{k \in V} \limsup_{\epsilon \searrow 0}  \int_{\RR}  \int_{\RR}  h_\epsilon (\lambda , \mu) \drm \langle P E_\lambda  \psi , \delta_k \rangle   \drm \langle \delta_k , P E_\mu  \psi \rangle ,
\end{align*}
where
\[
 h_\epsilon (\lambda , \mu) = 2 \epsilon \frac{2 \epsilon + \i (\lambda - \mu)}{4 \epsilon^2 + (\lambda - \mu)^2} \chi_{(a,b)}(\lambda)\chi_{(a,b)}(\mu) .
\]
Obviously, $\lim_{\epsilon \searrow 0} h_\epsilon (\lambda , \mu) = \delta_{\lambda \mu} \chi_{(a,b)} (\lambda)$ for all $\lambda,\mu \in \RR$.
Since $\lvert h_\epsilon (\lambda , \mu) \rvert \leq 1$ for all $\epsilon > 0$ and all $\lambda,\mu \in \RR$, we obtain by polarization and the dominated convergence theorem
\[
 L \leq \sum_{k \in V} \int_{\RR}  \int_{\RR} \delta_{\lambda \mu} \chi_{(a,b)} (\lambda) \drm \langle P E_\lambda  \psi , \delta_k \rangle   \drm \langle \delta_k , P E_\mu  \psi \rangle ,
\]
which gives $L \leq R$.
\end{proof}
\begin{proof}[Proof of Theorem \ref{theorem:result2} (spectral localization)]
Let $\mathcal{H}_p = \mathcal{H}_p (H_\omega)$ and $\mathcal{H}_c = \mathcal{H}_c (H_\omega)$ be the pure point and the continuous subspace of $\ell^2 (V)$ with respect to $H_\omega$. Let further $E_c = E_c (H_\omega)$ be the orthogonal projection onto $\mathcal{H}_c$. Let $(a,b) =: I \subset \RR$ be an arbitrary bounded interval. We will show that for almost all $\omega \in \Omega$, $P_I \delta_k \in \mathcal{H}_p$ for all $k \in V$. This implies that the spectrum in $(a,b)$ is almost surely only of pure point type. Since a countable intersection of sets with full measure has full measure, we can conclude that $H_\omega$ has almost surely no continuous spectrum.
\par
Fix a vertex $o \in V$. The criterion of Ruelle characterizes states from the subspace $\mathcal{H}_c$ as such states, which leave all compact sets in time mean \cite{Ruelle1969,CyconFKS1987}. More precisely, for all $\phi \in \ell^2 (V)$ we have
\begin{equation} \label{eq:ruelle}
 \norm{E_c \phi}^2 = \lim_{R \to \infty} \lim_{t \to \infty} \frac{1}{t} \int_0^t \norm{P_{R,o} {\rm e}^{-{\rm i}sH_\omega} \phi}^2 \drm s ,
\end{equation}
where $P_{R,o} : \ell^2 (V) \to \ell^2 (V)$ denotes the projection on states which vanish in $\{k \in V : d(k,o) < R\}$, i.\,e.
\[
\bigl(P_{R,o} \psi \bigr) (k) = \begin{cases}
                                 0 & \text{if $d(k,o)<R$,}\\
				 \psi (k) & \text{else.}
                                \end{cases}
\]
For non-negative functions $f$ we have the inequality \cite{FroehlichS1983}
 \[
\frac{1}{t} \int_0^t f(s) \drm s = 2\epsilon \int_0^{\frac{1}{2\epsilon}} f(s) \drm s \leq 2 \epsilon \int_0^{\frac{1}{2\epsilon}} e^{1-2\epsilon s}f(s) \drm s \leq  2 \epsilon \int_0^{\infty} e^{1-2\epsilon s}f(s) \drm s .
\]
We thus obtain from Eq. \eqref{eq:ruelle} and Lemma \ref{lemma:graf} for arbitrary $\psi \in \ell^2 (V)$
 \begin{align*}
 \norm{E_c P_I \psi}^2 &\leq \lim_{R \to \infty} \lim_{\epsilon \searrow 0} 2\epsilon \int_0^\infty \euler^{1-2\epsilon s} \norm{P_{R,o} {\rm e}^{-{\rm i}sH_\omega} P_I \psi}^2 \drm s \\
&\leq \lim_{R \to \infty} \liminf_{\epsilon \searrow 0} \frac{\euler \epsilon}{\pi}  \int_I \norm{P_{R,o} (H_\omega - E - \i\epsilon)^{-1} \psi}^2 \drm E .
\end{align*}
We choose $\psi = \delta_y$ with $y \in V$, take expectation, use Parseval's identity, Fatou's lemma and the monotone convergence theorem, and obtain for all $y \in V$
\begin{align*}
 \EE{\norm{E_{c} P_I \delta_y}^2} &\leq \mathbb{E}\Biggl\{\lim_{R \to \infty} \liminf_{\epsilon \searrow 0} \frac{\epsilon \euler}{\pi} \int_I  \sum_{\{k \in V : d(o,k) \geq R\}} \abs{G_\omega (E+\i \epsilon ; k,y)}^2 \drm E\Biggr\} \\[1ex]
& \leq \liminf_{R \to \infty} \liminf_{\epsilon \searrow 0} \frac{\epsilon \euler}{\pi} \int_I  \sum_{\{k \in V : d(o,k) \geq R\}} \mathbb{E} \Bigl\{\abs{G_\omega (E+\i \epsilon ; k,y)}^2\Bigr\} \drm E .
\end{align*}
By Lemma \ref{lemma:s=2} we have $\abs{\im \epsilon} \EE \bigl\{ \lvert G_\omega (E + \i \epsilon;k,y) \rvert^2 \bigr\} \leq  \max\{1,\pi\norm{\rho}_\infty\} C' C^{d(k,y)}$ $c_k(d($ $k,y))$ for all $k,y \in V$. Thus we have for all $y \in V$
\begin{align*}
\EE{\norm{E_{c} P_I \delta_y}^2} &\leq 
\frac{\max\{1 , \pi \norm{\rho}_\infty \}}{\pi (C' \abs{I} \euler)^{-1}}  \liminf_{R \to \infty}  \sum_{\{k \in V : d(k,o) \geq R\}} C^{d(k,y)} c_k(d(k,y)) .
\end{align*}
Since Assumption \ref{ass:spectral} is satisfied and $C < \alpha^*$ by assumption, the last sum converges. Hence, the limes inferior equals zero which gives $\mathbb{E} \bigl\{ \Vert E_{c} P_I \delta_y \Vert ^2 \bigr\} = 0$ for all $y \in V$. We conclude for all $y \in V$ that $P_I \delta_y \in \mathcal{H}_p$ for almost all $\omega \in \Omega$. This in turn gives the statement of the theorem.
\end{proof}
%
% %
% %
% %%%%%%%%%%%%%%%%%%%%%%%%%%%%%%%%%%%%%%%%%%%%%%%%%%%%%%%%%%%%%%%%%%%%%%%%%%%%%%%%%%%
% %----------------------------------------------------------------------------------
% %         D Y N A M I C A L     L O C A L I Z A T I O N
% %----------------------------------------------------------------------------------
% %%%%%%%%%%%%%%%%%%%%%%%%%%%%%%%%%%%%%%%%%%%%%%%%%%%%%%%%%%%%%%%%%%%%%%%%%%%%%%%%%%%
% %
% % 
%
\subsection{Dynamical localization; proof of Theorem \ref{theorem:result3}} \label{sec:dyn_loc}
In this subsection we prove Theorem \ref{theorem:result3} on dynamical localization. For this purpose we will follow the line of \cite{HamzaJS2009,Stolz2010} and use a variant of Stone's formula formulated in Lemma \ref{lemma:stolz}. The relevance of certain variants of Stone's formula for proving dynamical localization is well known, see e.\,g. \cite{AizenmanG1998,HamzaJS2009}. However, in our setting it is interesting that a factor one half has to be taken into account. Notice that Eq. \eqref{eq:stolz} is no longer valid if the left hand side is replaced by $\langle \psi , f(H) P_{[a,b]} \phi \rangle$ or $\langle \psi , f(H) P_{(a,b)} \phi \rangle$.
\begin{lemma} \label{lemma:stolz}
Let $H$ be a selfadjoint operator on a Hilbert space $\mathcal{H}$, $P_I = P_I (H)$ be the spectral projection onto the interval $I \subset \RR$ associated to the operator $H$. Let further $f : \RR \to \CC$ be a bounded continuous function and $a,b \in \RR$ with $a<b$. Then,
\begin{multline} \label{eq:stolz}
\frac{1}{2}\sprod{\psi}{f(H) (P_{(a,b)} + P_{[a,b]}) \phi} \\ = \lim_{\epsilon \searrow 0} \frac{\epsilon}{\pi} 
\int_a^b f(E) \sprod{\psi}{(H - E - \ii \epsilon)^{-1}(H - E + \ii \epsilon)^{-1}  \phi} \drm E .
\end{multline}
\end{lemma}
\begin{proof}
We prove Eq. \eqref{eq:stolz} in the case $\psi = \phi$. The case where $\psi \not = \phi$ then follows by polarization. We denote by $\{ E_\lambda \}_{\lambda \in \RR}$ the spectral family associated to the operator $H$ and $\drm \mu_\psi (\lambda) := \drm \! \sprod{\psi}{E_\lambda \psi}$. Notice that the measure $\drm \mu_\psi (\lambda)$ is a positive and finite Borel measure.
Using the spectral theorem and Fubini's theorem, we obtain for the right hand side of Eq. \eqref{eq:stolz}
\begin{align*}
 R &= \lim_{\epsilon \searrow 0} \frac{\epsilon}{\pi}  \int_a^b f(E) \left(  \int_\RR \frac{\drm \mu_\psi (\lambda)}{(\lambda - E)^2 + \epsilon^2} \right) \drm E 
     =  \lim_{\epsilon \searrow 0} \int_\RR  g_\epsilon (\lambda)  \drm \mu_\psi (\lambda) ,
\end{align*} 
where
\[
 g_\epsilon (\lambda) = \frac{\epsilon}{\pi} \int_\RR \frac{\chi_{[a,b]} (E) f(E)}{(\lambda - E)^2 + \epsilon^2} \drm E .
\]
We infer from Lemma \ref{lemma:approx_ident} that 
\[
 \lim_{\epsilon \searrow 0}  g_\epsilon (\lambda) = \frac{f(\lambda)}{2}  \left( \, \chi_{[a,b]} (\lambda) +  \chi_{(a,b)} (\lambda) \right) 
\]
for all $\lambda \in \RR$. Further, for all $\epsilon > 0$ and $\lambda \in \RR$ we have
\begin{align*}
 \abs{g_\epsilon (\lambda)} &\leq \lVert f \rVert_\infty \frac{\epsilon}{\pi} \int_a^b \frac{\drm E}{(\lambda - E)^2 + \epsilon^2} = \frac{\lVert f \rVert_\infty}{\pi} \left[\arctan\Biggl(\frac{\lambda - a}{\epsilon}\Biggr) - \arctan \Biggl(\frac{\lambda - b}{\epsilon}\Biggr) \right] \\ & \leq \lVert f \rVert_\infty .
\end{align*}
Using the dominated convergence theorem we obtain
\[
 R =  \frac{1}{2}\int_\RR f(\lambda) \left( \chi_{[a,b]} (\lambda) +  \chi_{(a,b)} (\lambda) \right)   \drm \mu_\psi (\lambda) = \frac{1}{2} \sprod{\psi}{f(H) (P_{(a,b)} + P_{[a,b]}) \psi} ,
\]
which ends the proof.
\end{proof}
Notice that the case $f \equiv 1$ of Lemma \ref{lemma:stolz} corresponds to Stone's formula, see e.\,g. \cite[Theorem 4.3]{Teschl2009}. From Lemma~\ref{lemma:stolz} one concludes the well known observation that for any $g \in C_{\rm c} ((a,b))$
\begin{equation*}
\sprod{\psi}{g(H) \phi} = \lim_{\epsilon \searrow 0} \frac{\epsilon}{\pi} \int_a^b g(E) \sprod{\psi}{(H - E - \ii \epsilon)^{-1}(H - E + \ii \epsilon)^{-1}  \phi}\drm E .
\end{equation*}
Hence, by Lebesgue's theorem
\begin{multline} \label{eq:stone2}
 \sup_{t \in \RR} \bigl\lvert \bigl\langle \psi ,  \euler^{-\i t H} P_{(a,b)} \phi \bigr\rangle \bigr\rvert 
\leq \liminf_{\epsilon \searrow 0} \frac{\epsilon}{\pi} \int_a^b \bigl\lvert\sprod{\psi}{(H - E - \ii \epsilon)^{-1}(H - E + \ii \epsilon)^{-1}  \phi}\bigr\rvert \drm E .
\end{multline}
\begin{proof}[Proof of Theorem \ref{theorem:result3} (dynamical localization)]
% By Ineq. \eqref{eq:stone2}, Parseval's identity and the %triangle inequality we obtain for all $x,y \in V$
%
% \begin{multline*}
% \sup_{t \in \RR}\bigl\lvert \langle \delta_x , \euler^{-\i t %H_\omega} P_{(a,b)} \delta_y \rangle \bigr\rvert \\ 
%
% \leq \liminf_{\epsilon \searrow 0} \frac{\epsilon}{\pi} %\int_{a}^{b} \left( \sum_{k \in V} \bigl\lvert G_\omega (E+\i %\epsilon; x,k)\bigr\rvert \, \bigl\lvert G_\omega (E-\i %\epsilon ; k,y)\bigr\rvert \right)  \drm E .
% \end{multline*}
%
From Ineq. \eqref{eq:stone2} we conclude using Parseval's identity, Fatou's lemma, Fubini's theorem, Chauchy-Schwarz inequality and Lemma \ref{lemma:s=2} that for all $x,y\in V$
\begin{align}
\EE \Bigl\{ \sup_{t \in \RR} \bigl| \langle & \delta_x , \euler^{-\i t H_\omega} P_{(a,b)} \delta_y \rangle \bigr| \Bigr\} \nonumber \\
&\leq \liminf_{\epsilon \searrow 0} \frac{\epsilon}{\pi} \int_{a}^{b} \left( \sum_{k \in V} \EE \bigl\{ \bigl\lvert G_\omega (E+\i \epsilon; x,k)\bigr\rvert \, \bigl\lvert G_\omega (E-\i \epsilon ; k,y)\bigr\rvert \right) \bigr\}  \drm E \nonumber \\
&\leq \liminf_{\epsilon \searrow 0} \frac{\epsilon}{\pi} \int_{a}^{b} \sum_{k \in V} 
\sqrt{\EE \bigl\{ |G_\omega (E - \i\epsilon;x,k)|^2 \bigr\}}
\sqrt{\EE \bigl\{ |G_\omega (E + \i\epsilon;k,y)|^2 \bigr\}} \drm E  \nonumber\\
&\leq
\frac{\max\{\pi,\Vert \rho \Vert_\infty\} C'}{\pi(b-a)^{-1}} 
\sum_{k \in V} C^{\frac{d(x,k) + d(k,y)}{2}} c_{x}(d(x,k))^{1/2} c_{k}(d(k,y))^{1/2} \label{eq:dyn_decay},
\end{align}
where $C$ and $C'$ are the constants from Theorem \ref{theorem:result1}.
By Parseval's identity and the triangle inequality we obtain for all $o \in V$ and any $p \geq 0$ the estimate
\[
 \norm{\lvert X_o \rvert^p \euler^{- \i t H_\omega} P_{(a,b)} \psi} \leq \norm{\psi}_\infty \sum_{\genfrac{}{}{0pt}{2}{y \in V :}{y \in \supp \psi}} \sum_{x \in V}
\lvert d(o,x) \rvert^p \abs{\langle \delta_x , \euler^{-\i t H_\omega} P_{(a,b)} \delta_y \rangle} .
\]
Since $\psi$ is of compact support, is sufficient to show that for some $o \in V$, all $y \in \supp \psi$ and any $p \geq 0$, 
\[
E := \Bigl\{ \sup_{t \in \RR} \sum_{x \in V}
 \lvert d(o,x) \rvert^p \lvert \langle \delta_x , \euler^{-\i t H_\omega} P_{(a,b)} \delta_y \rangle \rvert \Bigr\} < \infty \quad \text{almost surely}.
\]
By Ineq. \eqref{eq:dyn_decay}, Fubini's theorem and Assumption \ref{ass:dyn}, there is an $o \in V$ such that $\EE\{E\} < \infty$ for all $y \in V$ and any $p \geq 0$. This implies $E < \infty$ almost surely for all $y \in \supp \psi$ and any $p \geq 0$.
\end{proof}
%
% %
% %
% %%%%%%%%%%%%%%%%%%%%%%%%%%%%%%%%%%%%%%%%%%%%%%%%%%%%%%%%%%%%%%%%%%%%%%%%%%%%%%%%%%%
% %----------------------------------------------------------------------------------
% %         A P P E N D I X
% %----------------------------------------------------------------------------------
% %%%%%%%%%%%%%%%%%%%%%%%%%%%%%%%%%%%%%%%%%%%%%%%%%%%%%%%%%%%%%%%%%%%%%%%%%%%%%%%%%%%
% %
% % 
%

%\begin{appendix}
\section{On the Assumptions \ref{ass:spectral} and \ref{ass:dyn}} \label{sec:appendix}
In this section we give proofs for the statements (i)-(iii) given in Remark~\ref{remark:assumptions}. 
To prove statement (i) we use the fact that $\lvert S_x (n) \rvert \leq c_x (n) \leq K(K-1)^{n-1}$ for all $x \in V$ and $n \in \NN$, where $K$ denotes the uniform bound on the vertex degree. With the help of this inequality we have for all $y \in V$
\begin{align*}
 \sum_{k \in V} \alpha^{d(k,y)} c_k (d(k,y)) \leq  \sum_{i=0}^\infty \lvert S_y (i) \rvert \alpha^i \sup_{k \in S_y(i)} c_k (i) 
 \leq \sum_{i=0}^\infty \frac{K^2}{(K-1)^2} \left[ \alpha (K-1)^2 \right]^i
\end{align*}
which is finite if $\alpha$ is sufficiently small. Hence Assumption \ref{ass:spectral} is satisfied if the graph $G$ has uniformly bounded vertex degree. A similar calculation shows that Assumption \ref{ass:dyn} is satisfied if the graph $G$ has uniformly bounded vertex degree.
\par
Statement (ii) of Remark \ref{remark:assumptions} concerns examples of graphs which do not have a uniform bound on the vertex degree but satisfy Assumptions \ref{ass:spectral} and \ref{ass:dyn}. 
\begin{example}\label{ex:T}
We denote by $T = (V,E)$ the rooted and radial symmetric tree with vertex set $V$ and edges set $E$, where the number of offsprings $O(g)$ in generation $g \in \NN_0$ (the root corresponds to generation 0) is given by 
\[
 O (g) = \begin{cases}
          \log_2 g & \text{if $\log_2 g \in \NN$}, \\
		1  & \text{else}. 
         \end{cases}
\]
This tree has no uniform bound on the vertex degree. See Fig.~\ref{fig:T} for an illustration of this tree. 
\begin{figure}[ht] \centering
\begin{tikzpicture}[scale=0.7]
\draw (0,0)--(4,0);
\foreach \x in {0,1,2,3,4}
	\filldraw (\x,0) circle (2pt);
\draw (4,0)--(5,2);
\draw (4,0)--(5,-2);
\draw (8,2)--(9,3);
\draw (8,2)--(9,1);
\draw (8,-2)--(9,-3);
\draw (8,-2)--(9,-1);
\draw (5,2)--(11,2);
\draw (5,-2)--(11,-2);
\draw (9,3)--(11,3);
\draw (9,-3)--(11,-3);
\draw (9,1)--(11,1);
\draw (9,-1)--(11,-1);
\draw[dotted] (11,2)--(13,2);
\draw[dotted] (11,-2)--(13,-2);
\draw[dotted] (11,3)--(13,3);
\draw[dotted] (11,-3)--(13,-3);
\draw[dotted] (11,1)--(13,1);
\draw[dotted] (11,-1)--(13,-1);
\foreach \x in {5,6,7,8,9,10,11,13}
	\filldraw (\x,2) circle (2pt);
\foreach \x in {5,6,7,8,9,10,11,13}
	\filldraw (\x,-2) circle (2pt);
\foreach \x in {9,10,11,13}
	\filldraw (\x,3) circle (2pt);
\foreach \x in {9,10,11,13}
	\filldraw (\x,-3) circle (2pt);
\foreach \x in {9,10,11,13}
	\filldraw (\x,1) circle (2pt);
\foreach \x in {9,10,11,13}
	\filldraw (\x,-1) circle (2pt);
\foreach \x in {0,1,2,3,4,5,6,7,8,9,10,11}
	\draw (\x,-4) node {\x};
\draw (13,-4) node {16};
\draw (14,-4) node {17};
\draw (15,-4) node {18};
\foreach \x in {1,2,3,-1,-2,-3}{
	\draw (13,\x)--(14,\x+0.125);
	\draw (13,\x)--(14,\x+0.375);
	\draw (13,\x)--(14,\x-0.125);
	\draw (13,\x)--(14,\x-0.375);
}
\foreach \z in {14,15}
\foreach \x in {1,2,3,-1,-2,-3}
	\foreach \y in {0.125,0.375,-0.125,-0.375}
		\filldraw (\z,\x+\y) circle (2pt);
\foreach \x in {1,2,3,-1,-2,-3}
	\foreach \y in {0.125,0.375,-0.125,-0.375}
\draw (14,\x+\y)--(15,\x+\y);
\foreach \x in {1,2,3,-1,-2,-3}
	\foreach \y in {0.125,0.375,-0.125,-0.375}
\draw[dotted] (15,\x+\y)--(16,\x+\y);
\end{tikzpicture}
\caption{Illustration of the rooted tree from Example \ref{ex:T} and Remark \ref{remark:assumptions}} 
\label{fig:T}
\end{figure}
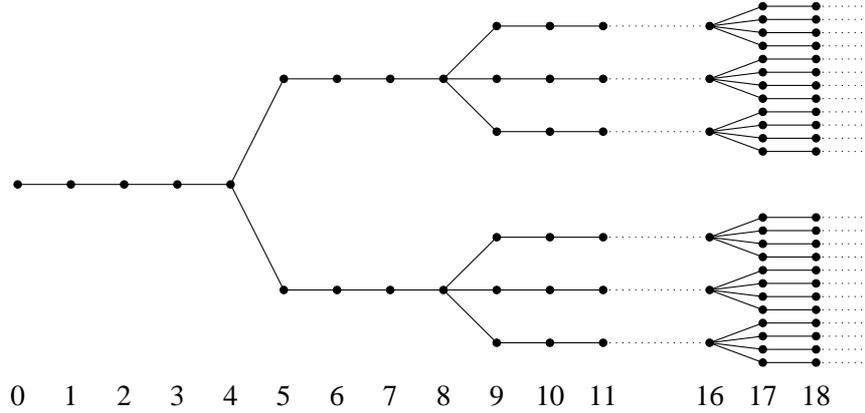
For $x \in V$ we denote by $g(x)$ the generation of the vertex $x$.
In order to show that this tree satisfies Assumptions~\ref{ass:spectral} and \ref{ass:dyn} it is essential to have an estimate on the self-avoiding walks. Since a self-avoiding walk with length $n$ starting at $x$ has at most $\lfloor \log_2 (g(x) + n-1) \rfloor$ neighbors to choose in a single step, we have
\[
 c_x (n) = \left \lfloor \log_2 (g(x) + n - 1) \right\rfloor !
\]
for all $x \in V$ and all $n \in \NN$. Hence we have using $\lvert S_y (n) \rvert \leq c_y (n)$ for all $y \in V$
\begin{align*}
 \sum_{k \in V} \alpha^{d(k,y)} c_k (d(k,y)) &\leq \sum_{i=0}^\infty \alpha^i \lvert S_y (i) \rvert \sup_{k \in S_y (i)} c_k (i) \\
&\leq \sum_{i=0}^\infty \alpha^i \left \lfloor \log_2 (g(y)+i-1) \right \rfloor ! \left \lfloor \log_2 (g(y)+2i-1) \right \rfloor  ! \, ,
\end{align*}
which is finite if $\alpha < 1$, by Stirling's formula and since $x \mapsto (\ln x)^{\ln x}$ grows slower than every exponential function. Hence Assumption \ref{ass:spectral} is satisfied and $\alpha^* = 1$. A similar calculation shows for all $y \in V$ and any $p \geq 0$
\begin{multline*}
 \sum_{x \in V}\sum_{k \in V} \lvert d(o,x) \rvert^p \Bigl(\beta^{d(x,k) + d(k,y)} c_{x}(d(x,k)) c_{k}(d(k,y))\Bigr)^{\frac{1}{2}}\\ 
\leq \sum_{i,j=0}^\infty \beta^{\frac{i+j}{2}} (g(y) + i + j)^p \left( \lfloor \log_2 (g(y) + 2i+2j) \rfloor ! \right)^2 ,
\end{multline*}
where $o$ is the root of the tree. Again by Stirling's formula the sum is convergent if $\beta < 1$. Hence Assumption \ref{ass:dyn} is satisfied and $\beta^* = 1$.
\end{example}
\begin{example}\label{ex:Z}
In this example we consider the graph $\mathcal{G} = (V,E)$ with vertex set $V = \ZZ^d$ and edges set $E$. Two vertices are connected by an edge if their $\ell^1$-distance equals one and for $n = 3,4,5,\ldots$ the vertex $(2^n , 0)$ is connected to all vertices whose $\ell^1$-distance to $(2^n , 0)$ equals $n$. This graph has no uniform bound on the vertex degree. See Fig. \ref{fig:Z} for an illustration of the graph $\mathcal{G}$.
\begin{figure}[ht] \centering
\begin{tikzpicture}[scale=0.33]
\foreach \x in {5,6,7,8,9,10,11,12,13,14,15,16,17,18,19,20,21,22,23,24,25,26,27,28,29,30,31,32,33,34,35,36,37}
	\foreach \y in {-5,-4,-3,-2,-1,0,1,2,3,4,5}
		\filldraw (\x,\y) circle (3pt);
\foreach \x in {5,6,7,8,9,10,11,12,13,14,15,16,17,18,19,20,21,22,23,24,25,26,27,28,29,30,31,32,33,34,35,36}
	\foreach \y in {-5,-4,-3,-2,-1,0,1,2,3,4,5}
	\draw (\x,\y) -- (\x+1,\y);
\foreach \x in {5,6,7,8,9,10,11,12,13,14,15,16,17,18,19,20,21,22,23,24,25,26,27,28,29,30,31,32,33,34,35,36,37}
	\foreach \y in {-5,-4,-3,-2,-1,0,1,2,3,4}
	\draw (\x,\y) -- (\x,\y+1);
\draw (8,0)--(9,2);
\draw (8,0)--(9,-2);
\draw (8,0)--(7,2);
\draw (8,0)--(7,-2);
\draw (8,0)--(10,1);
\draw (8,0)--(10,-1);
\draw (8,0)--(6,1);
\draw (8,0)--(6,-1);
\draw (8,0) ..controls (9.5,0.5).. (11,0);
\draw (8,0) ..controls (6.5,-0.5).. (5,0);
\draw (8,0) ..controls (7.5,1.5).. (8,3);
\draw (8,0) ..controls (8.5,-1.5).. (8,-3);
\draw (16,0)--(17,3);
\draw (16,0)--(17,-3);
\draw (16,0)--(15,3);
\draw (16,0)--(15,-3);
\draw (16,0)--(19,1);
\draw (16,0)--(19,-1);
\draw (16,0)--(13,1);
\draw (16,0)--(13,-1);
\draw (16,0) ..controls (17,1.5).. (18,2);
\draw (16,0) ..controls (15,0.5).. (14,2);
\draw (16,0) ..controls (17,-0.5).. (18,-2);
\draw (16,0) ..controls (15,-1.5).. (14,-2);
\draw (16,0) ..controls (18,0.5).. (20,0);
\draw (16,0) ..controls (14,-0.5).. (12,0);
\draw (16,0) ..controls (15.5,2).. (16,4);
\draw (16,0) ..controls (16.5,-2).. (16,-4);
\draw (32,0)--(33,4);
\draw (32,0)--(34,3);
\draw (32,0)--(35,2);
\draw (32,0)--(36,1);
\draw (32,0)--(33,-4);
\draw (32,0)--(34,-3);
\draw (32,0)--(35,-2);
\draw (32,0)--(36,-1);
\draw (32,0)--(31,4);
\draw (32,0)--(30,3);
\draw (32,0)--(29,2);
\draw (32,0)--(28,1);
\draw (32,0)--(31,-4);
\draw (32,0)--(30,-3);
\draw (32,0)--(29,-2);
\draw (32,0)--(28,-1);
\draw (32,0) ..controls (31.5,2.5).. (32,5);
\draw (32,0) ..controls (34.5,0.5).. (37,0);
\draw (32,0) ..controls (32.5,-2.5).. (32,-5);
\draw (32,0) ..controls (29.5,-0.5).. (27,0);
\foreach \x in {5,6,7,8,9,10,11,12,13,14,15,16,17,18,19,20,21,22,23,24,25,26,27,28,29,30,31,32,33,34,35,36,37}
{
\draw[dotted] (\x,5)--(\x,6);
\draw[dotted] (\x,-5)--(\x,-6);
}
\foreach \y in {-5,-4,-3,-2,-1,0,1,2,3,4,5}
{
\draw[dotted] (37,\y)--(38,\y);
\draw[dotted] (5,\y)--(4,\y);

}
\end{tikzpicture}
\caption{Illustration of the graph from Example \ref{ex:Z} and Remark \ref{remark:assumptions}} 
\label{fig:Z}
\end{figure}
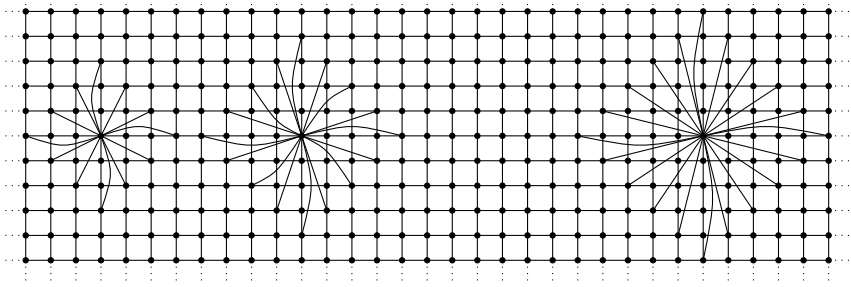
We start estimating $c_0 (n)$, i.\,e. the number of self-avoiding walks of length $n$ starting at $(0,0)$. In $n$ steps one can not get more in the east than $(2n,0)$. Hence the maximum vertex degree for the walk is $L_{\rm max} = 4(\lfloor \log_2 (2n) \rfloor+1)$. Since the walk is self-avoiding, a vertex with vertex degree larger than $5$ and a vertex with vertex degree equal to 5 can be visited only $\lfloor \log_2 (2n) \rfloor - 2$ times. Hence
\[
 c_0 (n) \leq 4 \cdot 3^{n-1}  \cdot (4L_{\rm max})^{\lfloor\log_2 (2n) \rfloor - 2} .
\]
Analogously we obtain for arbitrary $x \in V$
\[
 c_x (n) \leq 4 \cdot 3^{n-1}  \cdot (4K_{\rm max})^{\lfloor\log_2 (2n+\lvert x \rvert_1) \rfloor - 2},
\]
where
\[
  K_{\rm max} = 4(\lfloor \log_2 (2n + \lvert x \rvert_1) \rfloor+1) .
\]
The number of self-avoiding walks grows ``slowly enough'', such that a calculation similar to the one of Example \ref{ex:T} shows us that Assumptions \ref{ass:spectral} and \ref{ass:dyn} are satisfied for the graph $\mathcal{G}$.
\end{example}
\begin{remark}
 What we have learned from Example \ref{ex:T} and Example \ref{ex:Z} is the following. Let $G = (V,E)$ be a locally finite graph. If for each $y \in V$ there are constants $a,b>0$ such that $c_y (n) \leq a \cdot b^n$ for all $n \in \NN$, then the graph satisfies Assumption \ref{ass:spectral}. If there are constants $a,b > 0$ such that $c_y (n) \leq a \cdot b^n$ for all $n \in \NN$ and all $y \in V$, then the graph satisfies also Assumption \ref{ass:dyn}.
\end{remark}

%\par
Let us now discuss statement (iii) of Remark \ref{remark:assumptions}. Let us first define a class of graphs which Hammersley \cite{Hammersley1957} called crystals. A \emph{crystal} is a graph $G = (V,E)$ with an infinite set of vertices and bonds satisfying
\begin{enumerate}[(a)]
 \item Each vertex belongs to just a finite number of outlike classes.
 \item The number of bonds from (but not necessarily to) any atom is finite.
 \item If a subset of vertices either contains only finitely many vertices, or does not contain any vertex of at least one outlike class, then this subset contains a vertex from which a bond leads to some atom not in the subset.
\end{enumerate}
An \emph{outlike class} is defined as follows. Two vertices $x,y \in V$ are outlike if $c_x (n) = c_y (n)$ for all $n$. An outlike class is a class of pairwise outlike vertices. In the case of undirected graphs condition (b) means that the graph is locally finite. Hammersley proved in \cite{Hammersley1957} for crystals, that there exists a constant $\mu$, called the \emph{connective constant}, such that
\[
 \lim_{n \to \infty} c_x (n)^{\frac{1}{n}} = \mu \quad \text{for all $x \in V$}.
\]
Notice that crystals are assumed to have a finite number of outlike classes. This ensures that for all $\epsilon > 0$ there exists an $n_0 = n_0 (\epsilon)$ such that 
\[
c_y (n) \leq (\mu + \epsilon)^n
\]
for all $n \geq n_0$ and all $y \in V$. In particular, $n_0$ is independent of $y$.
For crystals we have the following
\begin{lemma}
 Let $G=(V,E)$ be a crystal.
\begin{enumerate}[(a)]
 \item Assume that for each vertex $y \in V$ there is a polynomial $p_y$ such that $\lvert S_y (n) \rvert \leq p_y (n)$ for all $n \in \mathbb{N}$. Then Assumption \ref{ass:spectral} is satisfied and $\alpha^* = 1 / \mu$. 
 \item Assume there is a polynomial $p$, such that $\lvert S_y (n) \rvert \leq p(n)$ for all $n \in \mathbb{N}$ and $y \in V$. Then Assumption \ref{ass:dyn} is satisfied and $\beta^* = 1/\mu$.
 \item Assume that there are $a \in (0,\infty)$ and $b \in (1,\infty)$ such that $\lvert S_y (n) \rvert \leq ab^n$ for all $y \in V$ and $n \in \mathbb{N}$. Then Assumptions \ref{ass:spectral} and \ref{ass:dyn} are satisfied and $\alpha^* , \beta^* \geq 1/(b \mu)$. 
\end{enumerate}
\end{lemma}
\begin{proof}
 First we prove (a). Let $\alpha < 1/\mu$. Then there is $\epsilon_0 > 0$ such that $\alpha = 1/(\mu + \epsilon_0)$. Let $\epsilon \in (0,\epsilon_0)$. By assumption, there is an $n_0 \in \NN$ such that for all $y \in V$  
\begin{align*}
 \sum_{k \in V} \alpha^{d(k,y)} c_k(d(k,y)) &\leq \sum_{i=0}^\infty \alpha^i p_y (i) \sup_{k \in S_y (i)} c_k (i) \\
&\leq \sum_{i=0}^{n_0}\alpha^i p_y (i) \sup_{k \in S_y (i)} c_k (i) + 
       \sum_{i=n_0 + 1}^{\infty} \left( \frac{\mu + \epsilon}{\mu+\epsilon_0} \right)^i p_y (i) ,
\end{align*}
which is finite since $\epsilon < \epsilon_0$. This proves (a). To prove (b) let $\beta < 1/\mu$. Let $\epsilon_0 > 0$ such that $\beta = 1/(\mu + \epsilon_0)$ and $\epsilon \in (0,\epsilon_0)$. Then there is $n_0 \in \NN$ such that $c_x (n) \leq (\mu + \epsilon)^{n}$ for all $x \in V$ and $n \geq n_0$. We have for all $y \in V$
\begin{multline*}
 \sum_{x \in V}\sum_{k \in V} \lvert d(o,x) \rvert^p \Bigl(\beta^{d(x,k) + d(k,y)} c_{x}(d(x,k)) c_{k}(d(k,y))\Bigr)^{\frac{1}{2}} \\
\leq \sum_{i,j=0}^\infty p(i)p(j) \left(\beta^j \sup_{x \in S_y (j)} c_x (j)\right)^{1/2} (d(o,y)+i+j)^p \left(\beta^i \sup_{x \in B_y (j+i)} c_x (i)\right)^{1/2} ,
\end{multline*}
where $B_y (n) = \{k \in V : d(y,k) \leq n\}$. If $i \geq n_0$ we have $c_x (i)\leq (\mu + \epsilon)^{i}$, hence the sum converges, which proves (b). The proof of (c) is similar. What differs is that the spheres can only be estimated by exponential functions. As a consequence the convergence of the sums can only be ensured if $\alpha$ and $\beta$ are chosen smaller than $1/(\mu b)$.
\end{proof}
In particular $\ZZ^d$ is a crystal and the spheres of $\ZZ^d$ grow polynomially. Hence one can replace the term $2d$ in Theorem \ref{theorem:aizenman} by $\mu$. Notice that $\mu$ is typically smaller than $2d$, see e.\,g. \cite{HaraSS1993}.
\section*{Acknowledgements}
It is a pleasure to me to thank my advisor Ivan Veseli\'c for the suggestion to pursue this research direction and advice during this work. I would also like to thank Matthias Keller, Fabian Schwarzenberger and G\"u{}nter Stolz for useful discussions.
%
%\bibliographystyle{amsalpha}
%\bibliography{lit}   % name your BibTeX data base
%
\providecommand{\bysame}{\leavevmode\hbox to3em{\hrulefill}\thinspace}
\providecommand{\MR}{\relax\ifhmode\unskip\space\fi MR }
\providecommand{\MRhref}[2]{%
  \href{http://www.ams.org/mathscinet-getitem?mr=#1}{#2}
}
\providecommand{\href}[2]{#2}

\end{document}